\documentclass{amsart}

\usepackage{amsmath,amsfonts,amssymb,amsthm,epsfig,amscd,}
\usepackage[]{fontenc}
\usepackage{enumerate}
\usepackage[cmtip,all]{xy}
\usepackage{tabularx,multirow}

\usepackage[]{hyperref}
\usepackage[capitalise]{cleveref}
\usepackage{braket}
\usepackage{colonequals}

\usepackage{tikz}
\usepackage{tikz-cd}
\usetikzlibrary{calc,intersections,decorations.pathmorphing}

\usepackage{microtype}

\newtheorem{theorem}{Theorem}[section]
\newtheorem{conjecture}[theorem]{Conjecture}
\newtheorem{corollary}[theorem]{Corollary}
\newtheorem{lemma}[theorem]{Lemma}
\newtheorem{proposition}[theorem]{Proposition}

\newtheorem{alphatheorem}{Theorem}

\newtheorem{alphacorollary}[alphatheorem]{Corollary}
\newtheorem{alphaquestion}[alphatheorem]{Question}

\theoremstyle{definition}
\newtheorem{definition}[theorem]{Definition}

\theoremstyle{remark}
\newtheorem{remark}[theorem]{Remark}

\numberwithin{equation}{section}

\mathchardef
\mhyphen="2D

\newcommand\dash{\nobreakdash-\hspace{0pt}}
\newcommand\hilbn[2]{\ensuremath{#2^{[#1]}}} 
\newcommand\bounded{\ensuremath{\mathrm{b}}}
\newcommand\pg{\ensuremath{p_{\mathrm{g}}}}
\newcommand\qq{\ensuremath{q}}

\DeclareMathOperator\AJ{AJ}
\DeclareMathOperator\Bl{Bl}
\DeclareMathOperator\Bs{Bs}
\DeclareMathOperator\derived{\mathbf{D}}
\DeclareMathOperator\gon{gon}
\DeclareMathOperator\HH{H}
\DeclareMathOperator\hh{h}
\DeclareMathOperator\Hom{Hom}
\DeclareMathOperator\Jac{Jac}

\DeclareMathOperator\Perf{Perf}
\DeclareMathOperator\Pic{Pic}
\DeclareMathOperator\rank{rk}
\DeclareMathOperator\Supp{Supp}
\DeclareMathOperator\symmetric{\mathfrak{S}} 
\DeclareMathOperator\Sym{Sym}

\DeclareMathOperator\ntSOD{ntSOD}

\newcommand{\cB}{\ensuremath{\mathcal{B}}}

\newcommand{\cO}{\ensuremath{\mathcal{O}}}

\newcommand{\cX}{\ensuremath{\mathcal{X}}}

\newcommand{\bP}{\ensuremath{\mathbb{P}}}

\newcommand{\bfk}{\mathbf{k}}

\begin{document}

\title{Indecomposability of derived categories in families}
\author[F. Bastianelli]{Francesco Bastianelli}
\address{Università degli Studi di Bari Aldo Moro, Via Edoardo Orabona 4, 70125 Bari, Italy}
\curraddr{}
\email{francesco.bastianelli@uniba.it}
\thanks{F.B.~was partially supported by INdAM (GNSAGA)}

\author[P. Belmans]{Pieter Belmans}
\address{
  Mathematical Institute,
  Utrecht University,
  Budapestlaan 6,
  3584 CD Utrecht,
  The Netherlands
}
\curraddr{}
\email{p.belmans@uu.nl}
\thanks{}

\author[S. Okawa]{Shinnosuke Okawa}
\address{Department of Mathematics, Graduate School of Science, Osaka University, Machi\-kaneyama 1-1, Toyonaka, Osaka 560-0043, Japan}
\curraddr{}
\email{okawa@math.sci.osaka-u.ac.jp}
\thanks{S.O.~was partially supported by Grants-in-Aid for Scientific Research (16H05994, 16H02141, 16H06337, 18H01120, 20H01797) and the Inamori Foundation.}

\author[A. T. Ricolfi]{Andrea T.~Ricolfi}
\address{Scuola Internazionale Superiore di Studi Avanzati (SISSA), Via Bonomea 265, 34136 Trieste, Italy}
\curraddr{}
\email{aricolfi@sissa.it}
\thanks{}

\date{}

\dedicatory{}

\begin{abstract}
  Using the moduli space of semiorthogonal decompositions in a smooth projective family,
  introduced by the second, the third and the fourth author,
  we propose a novel approach to indecomposability questions for derived categories.
  Modulo a natural conjecture on the structure of the moduli space,
  we give both general results,
  and discuss interesting explicit examples of the behaviour of indecomposability in families,
  by relating it to the behaviour of the canonical base locus in families.
  These examples are symmetric powers of curves,
  certain regular surfaces of general type with large canonical base locus,
  and Hilbert schemes of points on surfaces.
  Indecomposability for symmetric powers of curves has been settled via other means,
  the other cases remain open
  and we expect that our analysis of the base locus
  will prove instrumental in finding unconditional proofs.
\end{abstract}

\maketitle

\section{Introduction}
One of the important tools in the study of derived categories of coherent sheaves on smooth projective varieties
is the notion of a semiorthogonal decomposition, which allows one to decompose the derived category into smaller pieces.
Kuznetsov's ICM addressess \cite{MR3728631,MR4680279} are excellent surveys of the topic.


In this paper we study which smooth projective varieties are ``atomic'',
in the sense that their derived categories cannot be decomposed.
To understand the importance of indecomposability, the relevant points are that

\begin{enumerate}
  \item
    \label{item_1}
    semiorthogonal decompositions often reflect properties of the geometry of the variety,
    e.g.,~by relating the variety to other, easier varieties,
    or suggesting it is rational, see, e.g.,~\cite{MR3727501,MR2605171};
  \item
    \label{item_2}
    operations in the minimal model program are conjectured to give semiorthogonal decompositions,
    where the guiding principle is the DK-hypothesis,
    which predicts that K-equivalent varieties have equivalent derived categories,
    and that a K-inequality between varieties exhibits one derived category
    as an admissible subcategory of the other, see e.g.~\cite{MR1949787,MR2483950}.
\end{enumerate}
From the general picture in point \eqref{item_2} the derived category being indecomposable
is conjectured to imply minimality in the sense of the minimal model program.
But there exist several minimal varieties whose derived categories have a non-trivial semiorthogonal decomposition.
Hence it is an important question to understand when derived categories of minimal varieties are indecomposable.
In this paper we initiate the study of this indecomposability question
by looking at its (expected) behaviour in families,
both in general,
and in explicit examples.

The first two instances where indecomposability was known are:
\begin{itemize}
  \item varieties with trivial canonical bundle, by using the arguments in \cite{MR1651025};
  \item curves of genus~$\geq 2$, by \cite[Theorem~1.1]{MR2838062}.
\end{itemize}
In \cite{1508.00682v2} these results were generalised, and it was shown how the geometry of the base locus~$\Bs|\omega_X|$
of the canonical linear series --- we call it the \emph{canonical base locus} --- governs indecomposability,
because of its link to the Serre functor for~$\derived^\bounded(X)$.
In particular, it is shown that for a semiorthogonal decomposition~$\derived^\bounded(X)=\langle\mathcal{A},\mathcal{B}\rangle$,
either all objects of~$\mathcal{A}$ or all objects of~$\mathcal{B}$
are necessarily supported on~$\Bs|\omega_X|$.
One of the main consequences is the following.

\begin{theorem}[Kawatani--Okawa]
  \label{theorem:kawatani-okawa}
  Let~$X$ be a smooth projective variety.
  If for all connected components~$Z$ of~$\Bs|\omega_X|$
  there exists a Zariski open subset~$U\subseteq X$ containing~$Z$
  such that~$\omega_X|_U\cong\mathcal{O}_U$,
  then~$\derived^\bounded(X)$ is indecomposable.

  In particular, if~$\dim\Bs|\omega_X|\leq 0$ then~$\derived^\bounded(X)$ is indecomposable.
\end{theorem}

To bootstrap from this result
we will consider the moduli space of semiorthogonal decompositions
in a smooth projective family~$\mathcal{X}\to T$,
introduced in \cite{2002.03303v2},
and relate it to the behaviour of the canonical base locus in families.
The main result of op.~cit.~is that this moduli functor of semiorthogonal decompositions
is an algebraic space,
which is moreover \'etale over~$T$.

For our applications we will consider the subfunctor
parameterizing nontrivial semiorthogonal decompositions,
as in~\S8.4 of~op.~cit.
Assuming \cite[Conjecture~8.28]{2002.03303v2} (see also \cref{conjecture:ntsod})
the latter is again an algebraic space \'etale over~$T$,
and thus as explained by \cref{theorem:Zariski open}
the locus in~$T$ where~$\derived^\bounded(f^{-1}(t))$ admits a non-trivial semiorthogonal decomposition
is Zariski-open.

This is a natural conjecture,
with the unobstructedness of deformations of exceptional collections \cite{MR3050709}
as the main evidence.
The indecomposability results in this paper are conditional on this conjecture.
We will recall the setup and context in \Cref{section:indecomposability},
where we will also prove the following.

\begin{alphatheorem}
  \label{theorem:indecomposability}
  Let~$f\colon\mathcal{X}\to T$ be a smooth projective morphism, where~$T$ is an excellent scheme.
  If \cref{conjecture:ntsod} holds,
  and the subset
  \begin{equation}
    U\colonequals\bigl\{ t\in T \,\bigm|\, \text{$\derived^\bounded(f^{-1}(t))$ is indecomposable} \bigr\} \subseteq T
  \end{equation}
  is dense,
  then for all~$t\in T$ we have that~$\derived^\bounded(f^{-1}(t))$ is indecomposable.
\end{alphatheorem}
In other words,
indecomposability of derived categories can be extended from general to special fibres of a family $f\colon\mathcal{X}\to T$ as above.

\medskip

To find instances where such a method would be applicable,
observe that in characteristic \( 0 \), thanks to the invariance of genera,
the dimension of the canonical base locus $\dim\Bs|\omega_{f^{-1}(t)}|$
is an upper-semicontinuous function $T\to \mathbb{N}\cup \{-\infty\}$ (cf. \Cref{proposition:canonical-base-locus-upper-semicontinuous}).
Therefore, if $\dim\Bs|\omega_{f^{-1}(t_0)}|\leq 0$ for some $t_0\in T$,
then $\dim\Bs|\omega_{f^{-1}(t)}|\leq 0$ for general $t\in T$,
and hence $\derived^\bounded(f^{-1}(t))$ is indecomposable by \Cref{theorem:kawatani-okawa}.
Combining this fact with \Cref{theorem:indecomposability} we then deduce the following.

\begin{alphacorollary}
  \label{corollary:zero-dim-Bs}
  Consider the situation from \Cref{theorem:indecomposability},
  and assume moreover that~~$T$ is irreducible and of characteristic~0.
  If there exists a point~$t_0\in T$ such that~$\dim\Bs|\omega_{f^{-1}(t_0)}|\leq 0$,
  then~$\derived^\bounded(f^{-1}(t))$ is indecomposable for all~$t\in T$.
\end{alphacorollary}

We note that the assumption on the characteristic of the ground field is essential.
In \Cref{remark:Fakhruddin}, we show that \Cref{corollary:zero-dim-Bs} does not hold in general in positive or mixed characteristic.
We prove \Cref{theorem:indecomposability} and \Cref{corollary:zero-dim-Bs} in \Cref{section:indecomposability}.

In the light of the latter result, it seems natural to raise the following question
on the invariance under deformation of indecomposability of derived categories.
\begin{alphaquestion}
  Let~$f\colon\mathcal{X}\to T$ be a smooth projective morphism, where~$T$ is irreducible,
  and suppose that~$\derived^\bounded(f^{-1}(t_0))$ is indecomposable for some~$t_0\in T$.
  Under which additional conditions we can ensure that $\derived^\bounded(f^{-1}(t))$ is indecomposable for all $t\in T$?
\end{alphaquestion}

\medskip

Now we turn to specific classes of complex projective varieties,
and explain how one can study the indecomposability of their derived categories
by applying \Cref{theorem:indecomposability} and \Cref{corollary:zero-dim-Bs}.

In \Cref{section:symmetric-products}, we consider $n$-fold symmetric products $\Sym^nC$
of smooth projective curves $C$ of genus $g\geq 3$,
and we describe the canonical base locus $\Bs|\omega_{\Sym^nC}|$
in relation to the gonality $\gon(C)$ of the curve $C$ (cf.~\Cref{proposition:curve-base-locus}).
The same result has been proved independently in \cite{1807.10702v1},
and it is used to show that $\derived^\bounded(\Sym^nC)$ is indecomposable for any $n<\gon(C)$
(see \cite[Theorem~1.2]{1807.10702v1}).
Using \Cref{theorem:indecomposability} and the fact that the gonality of the general curve is $\lfloor (g+3)/2\rfloor-1$,
we achieve the following strengthening of this indecomposability result,
see Proposition~\ref{proposition:symmetric-power-indecomposability}.
\begin{alphacorollary}
  \label{corollary:sym}
  Let $C$ be a smooth complex projective curve of genus $g\geq 2$
  and let $\Sym^nC$ be its $n$-fold symmetric product.
  Then
  assuming \cref{conjecture:ntsod}
  the derived category $\derived^\bounded(\Sym^nC)$ is indecomposable for any $1 \leq n\leq\lfloor\frac{g+3}{2}\rfloor-1$.
\end{alphacorollary}
However, our techniques do not apply to the range $\lfloor\frac{g+3}{2}\rfloor\leq n\leq g-1$
as then the canonical base locus $\Bs|\omega_{\Sym^nC}|$ is always positive-dimensional.
The expected indecomposability for the derived category for~$n\leq g-1$ was
obtained recently by Lin~\cite{2107.09564} using the \emph{paracanonical} base locus.

\medskip

In \Cref{section:surfaces} we examine two classes of surfaces of general type.
On one hand we consider \emph{Horikawa surfaces}, i.e.~those surfaces on, or near, the Noether line,
having~$\mathrm{c}_1^2=2\pg-k$ with~$k=2,3,4$,
which have been classified in \cite{MR1573789,MR0424831,MR0460340,MR0501370,MR0517773}.
In particular, we extend the analysis in \cite[Example~4.6]{1508.00682v2},
by summarising in \cref{appendix:horikawa} the behaviour of the canonical base locus in families.
Thus, assuming \cref{conjecture:ntsod}
\emph{all} derived categories of Horikawa surfaces are indecomposable (see \Cref{theorem:horikawa-indecomposability}).

We will also study a series of examples introduced by Ciliberto \cite{MR0778862}
for discussing properties of the canonical ring of surfaces of general type.
These surfaces are obtained as the minimal model of certain double planes,
and they all have positive-dimensional canonical base locus.
We explain how to construct families where the dimension of the base locus jumps
in \cref{theorem:ciliberto-family},
and for which indecomposability of the derived category is unconditional for most members of the family.
This degeneration shows,
assuming \cref{conjecture:ntsod},
that they all have indecomposable derived categories (see \cref{corollary:ciliberto-indecomposability}).


\medskip

Finally, in \Cref{section:hilbert-schemes} we study indecomposability of the derived category of Hilbert schemes $S^{[n]}$ of points on a surface $S$.
In particular, after showing that the emptiness of $\Bs|\omega_{S}|$
implies the emptiness of $\Bs|\omega_{S^{[n]}}|$ (see \Cref{proposition:canonical-base-locus-hilbert-schemes}),
we discuss the indecomposability of $\derived^\bounded(S^{[n]})$,
for some of the surfaces considered in \Cref{section:surfaces}.

\medskip

To conclude, let us point out the following question (stated as a conjecture in \cite{1807.10702v1}), where a positive answer would imply the indecomposability for all the examples studied in this paper. We point out that given a smooth projective variety~$X$, the nefness of~$\omega_X$ is nothing but the minimality in the sense of the minimal model program, and the existence of a global section of the canonical bundle is reminiscent of, but weaker than, the conditions in \cite{1508.00682v2}.

\begin{alphaquestion}
  \label{question:indecomposability}
  Let~$X$ be a smooth projective variety. Is it true that,
  if~$\omega_X$ is nef and $\HH^0(X,\omega_X)\neq 0$,
  then~$\derived^\bounded(X)$ is indecomposable?
\end{alphaquestion}

This question seems completely out of reach at the moment.
To see that the converse implication does not hold,
observe that bielliptic surfaces are examples of varieties where~$\omega_X$ is nef,
but~$\HH^0(X,\omega_X)=0$, and their derived categories are indecomposable by \cite[Proposition~4.1]{1508.00682v2}.
Higher-dimensional examples can be found amongst hyperelliptic varieties \cite{2411.14814}.


\section{Extending indecomposability}
\label{section:indecomposability}
In this section we will discuss the proof of \Cref{theorem:indecomposability} and \Cref{corollary:zero-dim-Bs}.
First we recall the following.
\begin{definition}
  \label{definition:sod}
  Let~$\mathcal{T}$ be a triangulated category. A \emph{semiorthogonal decomposition of length~$n$} for~$\mathcal{T}$ is a sequence~$\mathcal{A}_1,\ldots,\mathcal{A}_n$ of full triangulated subcategories, such that
  \begin{description}
    \item[semiorthogonality] for all~$1\leq j<i\leq n$ and for all~$A_i\in\mathcal{A}_i$ and~$A_j\in\mathcal{A}_j$ we have
      \begin{equation}
        \Hom_{\mathcal{T}}(A_i,A_j)=0;
      \end{equation}

    \item[generation] the categories~$\mathcal{A}_1,\ldots,\mathcal{A}_n$ generate~$\mathcal{T}$, i.e.~the smallest triangulated subcategory of~$\mathcal{T}$ containing~$\mathcal{A}_1,\ldots,\mathcal{A}_n$ is~$\mathcal{T}$.
  \end{description}
  For a semiorthogonal decomposition of~$\mathcal{T}$
  we will use the notation
  \begin{equation}
    \mathcal{T}
    =
    \langle\mathcal{A}_1,\ldots,\mathcal{A}_n\rangle.
  \end{equation}
\end{definition}

We will only be concerned with semiorthogonal decompositions of the bounded derived category~$\derived^\bounded(X)$ of coherent sheaves on a smooth projective variety~$X$ (for which \cite{MR3728631} provides an excellent starting point), and~$T$\dash linear semiorthogonal decompositions of the category~$\Perf\cX$ of perfect complexes on the total space~$\cX$ of a smooth projective morphism~$f\colon\cX\to T$. The notion of~$T$\dash linearity is introduced in \cite[\S2.6]{MR2238172} and ensures that the semiorthogonal decompositions behave well in families.

\begin{definition}
  Let~$X$ be a smooth projective variety.
  We say that~$\derived^\bounded(X)$ is \emph{indecomposable}
  if there exist no non-trivial semiorthogonal decompositions,
  i.e.,~in \Cref{definition:sod} all but one of the subcategories are always zero.
\end{definition}

\textbf{Moduli spaces of semiorthogonal decompositions.} \quad
In \cite{2002.03303v2} a moduli space of semiorthogonal decompositions was constructed,
relative to a smooth projective family~$f\colon\cX\to T$. The main result \cite[Theorem~A]{2002.03303v2}, under the assumption that
\(
   T
\)
is an excellent scheme, is that there exists an \'etale algebraic space
\(
    \operatorname{SOD}_f ^{ 2 }
    \to
    T
\)
which classifies semiorthogonal decompositions of length~2 of the bounded derived category of coherent sheaves on the fibres of~$f$.

Because we are only interested in indecomposability results,
we want to restrict ourselves to~$\operatorname{ntSOD}_f^2$:
the subfunctor defined by
considering only non-trivial semiorthogonal decompositions of length~2,
as introduced in \cite[Definition~8.24]{2002.03303v2}.
By \cite[Proposition~8.33]{2002.03303v2}, the~$T'$\dash valued points for~$T'\to T$ are given by
\begin{equation}
  \operatorname{ntSOD}_f^2(T')
  \simeq
  \left\{
    \text{$\Perf\cX_{T'}=\langle\mathcal{A}_1,\mathcal{A}_2\rangle$ $T'$-linear}\left|\begin{array}{l}
    \mathcal{A}_i|_x\neq 0\textrm{ for any}\\ i=1,2 \textrm{ and } x\in T' \end{array}
  \right.\right\},
\end{equation}
where~$\mathcal{A}_i|_x$ denotes the base change of the subcategory to a point~$x\in T'$.

We shall work assuming the following conjecture, implied by \cite[Conjecture~8.28]{2002.03303v2}, which in turn would follow from \cite[Conjecture~8.29]{2002.03303v2}.

\begin{conjecture}
  \label{conjecture:ntsod}
  Let~$f\colon\mathcal{X}\to T$ be a smooth projective morphism, where~$T$ is an excellent scheme.
  Then~$\ntSOD_f^2$ is an algebraic space, which is \'etale over~$T$.
\end{conjecture}

For us, the relevant consequence of this conjecture is the following.
\begin{theorem}\label{theorem:Zariski open}
  Granting \cref{conjecture:ntsod}, and under the same assumptions, 
  the set of points~$t \in T$ for which~$\derived^{\bounded} ( f ^{ - 1 } ( t ) )$
  admits a non-trivial semiorthogonal decomposition is a (possibly empty) Zariski open subset of \( T \).
\end{theorem}

\begin{proof}
  The subset of interest is nothing but the image of the \'etale (and thus open) morphism~$\ntSOD^{2}_{f} \to T$, hence is Zariski open.
\end{proof}

Note that, in particular, we can take \( T \) to be a scheme locally of finite type over~$\mathbb{C}$.

We can now obtain the following short proof of the main theorem.
\begin{proof}[Proof of \Cref{theorem:indecomposability}]
  By \cref{theorem:Zariski open}, the subset~\(T \setminus U \subseteq T\) is open.
  Since it does not intersect the dense subset~\(U \subseteq T\),
  it has to be empty.
  Hence~\(U =T\), which is nothing but the assertion.
\end{proof}

\textbf{Canonical base loci in families.} \quad
The main goal of the remainder of the paper
is to give instances in which we can apply this result,
for which we need to exhibit a dense subset~$U\subset T$ parameterising fibres of $f\colon\cX\to T$
with indecomposable derived category.
In this direction, we use \Cref{theorem:kawatani-okawa} and the following result.
It is likely standard, but we have not found a reference, so we include a proof.

\begin{proposition}
  \label{proposition:canonical-base-locus-upper-semicontinuous}
  Let~$f\colon\cX\to T$ be a smooth and projective morphism of noetherian schemes defined in characteristic \(0\).
  Then the function
  \begin{equation}
    \varphi \colon T \to \mathbb{N}\cup\set{-\infty},
    \quad
    t \mapsto \varphi ( t ) \colonequals \dim \Bs | \omega_{ \cX _{ t } }|
  \end{equation}
  is upper semi-continuous.
\end{proposition}

\begin{proof}
Define the closed subscheme
\(
  \cB \subset \cX
\),
which could be called the \emph{relative canonical base locus}, by the ideal sheaf~$\operatorname{im}(f^* f_*\omega_f\otimes\omega_f^{-1}\to\mathcal{O}_{\mathcal{X}})$, where the morphism is induced by tensoring the counit morphism~$f^* f_*\omega_f\to\omega_f$ by~$\smash{\omega_f^{-1}}$. Here~$\omega_f=\det\Omega_{\cX/T}^1$ is the relative canonical bundle.

We obtain the associated right exact sequence
\begin{equation}
  \label{eq:standard right exact sequence}
  f ^{ \ast } f _{ \ast }\omega _{ f } \otimes \omega _{ f } ^{ - 1}
  \to
  \cO _{ \cX }
  \to
  \cO _{ \cB }
  \to
  0.
\end{equation}
In the rest of the proof, we will show that for each
\(
  t \in T
\)
the fibre
\(
  \cB _{ t }
\)
coincides with the canonical base locus of the fibre
\(
  \cX _{ t } = f ^{ -1 } ( t )
\).
As the fibre dimension of the proper morphism
\(
  \cB \to T
\)
is upper semi-continuous, we obtain the assertion.

Taking the fibre of \eqref{eq:standard right exact sequence} over~$t\in T$ we obtain the right exact sequence
\begin{equation}
  \left( f _{ \ast }  \omega _{ f }  \otimes _{ \cO _{ T } } \bfk( t ) \right)
  \otimes _{ \bfk ( t ) } \omega _{ \cX _{ t } } ^{ - 1}
  \to
  \cO _{ \cX _{ t } }
  \to
  \cO _{ \cB _{ t } }
  \to
  0.
\end{equation}
Therefore all we have to show is that the canonical map
\begin{equation}
  f _{ \ast }  \omega _{ f }  \otimes _{ \cO _{ T } } \bfk ( t )
  \to
  \HH^{ 0 } ( \cX _{ t }, \omega _{ \cX _{ t } } )
\end{equation}
is an isomorphism. But this is a famous consequence of the degeneration of the Hodge-to-de Rham spectral sequence (see, say, \cite[Theorem 10.21]{MR1193913}) and this is where we need the assumption on the characteristic.
\end{proof}

\begin{remark}\label{remark:Fakhruddin}
It is well known that the invariance of (pluri-)genera fails in positive characteristic and mixed characteristic (and for non-K\"ahler manifolds). In fact, \Cref{proposition:canonical-base-locus-upper-semicontinuous} and \Cref{corollary:zero-dim-Bs} completely fail without the assumption on the characteristic. To see this, let
\(
  \bfk
\)
be an algebraically closed field of characteristic~\( 2 \), and
\(
  X
\)
be a supersingular Enriques surface \cite{MR3393362}.
Then one can find a smooth projective morphism
\(
  \cX \to T
\)
to a smooth pointed curve \( (T,0) \) such that
\begin{itemize}
  \item \(\cX _{ 0 } \simeq X\), and
  \item \( \cX _{ t } \) is a classical Enriques surface for $t \in T\setminus \set{0}$.
\end{itemize}

For a classical Enriques surface~$Y$ we have that
\(
  \HH ^{ i } ( Y,\cO_Y ) = 0
\)
for
\(
  i = 1, 2
\),
\(
  \omega_Y \not\simeq \cO_Y
\),
and
\(
  \omega_Y ^{ \otimes 2 }
  \simeq
  \cO_Y
\). On the other hand, for a singular or a supersingular Enriques surface~$X$ we have that
\(
  \omega_X \simeq \cO_X
\).
Therefore the family
\(
  \cX \to T
\)
considered above shows that the invariance of genera fails in characteristic \(2\).

Since
\(
  \omega _{ X } \simeq \cO _{ X }
\), there is no non-trivial semiorthogonal decomposition of
\(
  \derived ^{ \bounded } ( X )
\).
On the other hand, if we let
\(
  \pi
\)
be the restriction of the morphism $\mathcal X \to T$ over
\(
  T \setminus \set{ 0 }
\),
then the assumption implies that
\(
  \mathbf{R} \pi _{ \ast } \cO_{\mathcal{X}\setminus X} \simeq \cO_{T\setminus \set{ 0 }}
\).
Hence there is a non-trivial
\(
  (T \setminus \set{ 0 })
\)-linear semiorthogonal decomposition of
\(
  \derived ^{ \bounded } \left(\cX \setminus X\right)
\)
which does \emph{not} extend to the central fibre.
If we take
\(
  t _{ 0 } = 0
\)
in the statement of \Cref{corollary:zero-dim-Bs}, this shows that the assertion does not hold in characteristic \(2\).

It is shown in \cite[Theorem 4.10]{MR3393362} that any Enriques surface defined in positive characteristic admits a lifting to characteristic \(0\). Let
\(
  X
\)
be a supersingular Enriques surface defined in characteristic \(2\), and let
\(
  \cX \to T
\)
be its lifting to characteristic \(0\).
Since any Enriques surface in characteristic \(0\) is classical, this gives a similar example as above in mixed characteristic (where the characteristic of the residue field is \(2\)).
\end{remark}

We can now prove \cref{corollary:zero-dim-Bs}, which implements the explicit check for indecomposability in families.

\begin{proof}[Proof of \Cref{corollary:zero-dim-Bs}]
  Let~$\mathcal{X}\to T$ be a smooth projective morphism as in \Cref{theorem:indecomposability}. Then the set~$U\colonequals\varphi^{-1}(\set{-\infty,0})$ is an open subset of~$T$, by \Cref{proposition:canonical-base-locus-upper-semicontinuous}, which is non-empty by assumption. Applying \Cref{theorem:kawatani-okawa} we can conclude by \Cref{theorem:indecomposability}.
\end{proof}

\section{Symmetric products of curves}
\label{section:symmetric-products}
Let $C$ be a smooth complex projective curve of genus $g\geq 3$, and let $n\geq 1$.
We denote by $\Sym^nC$ the $n$-fold symmetric product, which is the smooth $n$-dimensional variety parameterising unordered $n$-tuples $D=p_1+\dots+p_n$ of points $p_1,\dots,p_n\in C$.
In this section we prove \Cref{corollary:sym} by applying \Cref{theorem:indecomposability}.
To this aim we need a good understanding of the canonical base locus of~$\Sym^nC$.
For this we recall some standard facts on the Abel--Jacobi map and Brill--Noether theory \cite[\S IV.1]{MR0770932}.

Let us consider the Abel--Jacobi map
\begin{equation}
\AJ_n\colon\Sym^nC\to\Pic^nC, \quad D \mapsto \mathcal O_C(D),
\end{equation}
such that~$\AJ_n^{-1}(\AJ_n(D))=\bP\HH^0(C,\mathcal{O}_C(D)) ^{ \vee }$. For~$r=0,\ldots,n$, we have the following standard subschemes of~$\Sym^nC$
\begin{equation}
  C_n^r
  \colonequals
  \left\{ D=q_1+\dots+q_n\mid \rank(\mathrm{d}\AJ_{n,D})\leq n-r \right\}\subseteq\Sym^nC,
\end{equation}
and the subschemes of the Picard scheme $W_n^r\colonequals \AJ_n(C_n^r)\subseteq\Pic^nC$, endowed with the morphisms
\begin{equation}
  \label{equation:AJ-n-r}
  \AJ_n^r=\AJ_n|_{C_n^r}\colon C_n^r\to W_n^r.
\end{equation}
We note that $\Sym^nC=C_n^0$, and we define~$W_n\colonequals W_n^0$.
Moreover, ignoring the scheme-theoretic structure, we have set-theoretic identities
\begin{equation}
  \begin{aligned}
    \Supp C_n^r
    &=
    \left\{ D\in\Sym^nC \mid \hh^0(C,\mathcal{O}_C(D))\geq r+1 \right\}, \\
    \Supp W_n^r
    &=
    \left\{ \mathcal{L}\in\Pic^nC \mid \hh^0(C,\mathcal{L})\geq r+1 \right\}.
  \end{aligned}
\end{equation}

Let $\gon(C)$ denote the \emph{gonality} of the curve $C$, that is the least degree of a non-constant morphism $C\to \mathbb{P}^1$.
The following result describes the canonical base locus of~$\Sym^nC$ depending on the gonality of $C$, and it also appears as \cite[Proposition~3.4]{1807.10702v1}.
We present an alternative proof, relying on Macdonald’s description of canonical divisors on $\Sym^n C$ \cite{MR0151460} and the geometric version of the Riemann--Roch theorem \cite[\S I.2]{MR0770932}.
\begin{proposition}
  \label{proposition:curve-base-locus}
  Let~$C$ be a non-hyperelliptic curve of genus~$g\geq 3$ and let~$n=1,\ldots,g-1$. Then
  \begin{enumerate}
    \item $\Bs|\omega_{\Sym^nC}|=\emptyset$ if and only if~$n<\gon(C)$;
    \item\label{enumerate:curve-base-locus-2} if~$n\geq\gon(C)$, then~$\Bs|\omega_{\Sym^nC}|$ is infinite, and set-theoretically we have an equality
      \begin{equation}
        \Bs|\omega_{\Sym^nC}|
        =
        C_n^1.
      \end{equation}
  \end{enumerate}
\end{proposition}

\begin{proof}
  It is well-known that~$\Bs|\omega_C|=\emptyset$, so we can assume that~$n=2,\ldots,g-1$.
  The following two lemmas show that the set-theoretic equality $\Bs|\omega_{\Sym^nC}|=C_n^1$ holds, so the proof ends by observing that~$C_n^1=\emptyset$ if and only if~$n<\gon(C)$.
\end{proof}

\begin{lemma}
  \label{lemma:base-locus-symn-C-1}
  If~$q_1+\cdots+q_n\in\Sym^nC$ is in~$\Bs|\omega_{\Sym^nC}|$ then~$q_1+\cdots+q_n\in C_n^1$.
\end{lemma}

\begin{proof}
  Let $\phi\colon C\hookrightarrow\mathbb{P}^{g-1}$ be the canonical embedding and let $P=p_1+\cdots+p_n\in \Sym^n C$. As in \cite[page~12]{MR0770932}, we denote by~$\overline{\phi(P)}$ the intersection of the hyperplanes~$H\subset \mathbb{P}^{g-1}$ such that~$\phi^*(H)\geqslant P$. In particular, we note that~$\dim \overline{\phi(P)}\leqslant n-1$, and if the points~$p_1,\ldots,p_n$ are distinct, then~$\overline{\phi(P)}$ is simply the linear span of~$\phi(p_1),\ldots,\phi(p_n)$.

  Given a~$(g-1-n)$-plane~$L\subset \mathbb{P}^{g-1}$, let~$\mathfrak{D}_L$ be the linear series of divisors~$\phi^*(H)$ on~$C$ cut out by the hyperplanes~$H\subset \mathbb{P}^{g-1}$ containing~$L$. So~$\mathfrak{D}_L$ is a linear system of degree $2g-2$ and dimension $n-1$, possibly, with base points at $L\cap \phi(C)$.

  Consider the subordinate variety~$\Gamma(\mathfrak{D}_L)$, which is a determinantal variety supported on the set
  \begin{equation}
    \Supp\Gamma(\mathfrak{D}_L)
    \colonequals
    \left\{ P=p_1+\cdots+p_n\in \Sym^n C\mid D-P\geqslant 0 \textrm{ for some }D\in |\mathfrak{D}_L| \right\}
  \end{equation}
  (cf.~\cite[pages~341--342]{MR0770932}). By \cite[Lemma~2.1]{MR2888217}, $\Gamma(\mathfrak{D}_L)$ is a canonical divisor of $\Sym^nC$.

  We claim that if~$Q=q_1+\cdots+q_n\in {\Sym^n C}$ is a base point of~$|\omega_{\Sym^n C}|$, then $\dim \overline{\phi(Q)}\leqslant n-2$. Indeed, for any~$(g-1-n)$-plane~$L\subset \mathbb{P}^{g-1}$, $Q$~lies on the canonical divisor~$\Gamma(\mathfrak{D}_L)$. Hence there exists a hyperplane~$H\subset \mathbb{P}^{g-1}$ containing~$L$ such that~$\phi^*(H)\geqslant Q$, so that~$\overline{\phi(Q)}\subset H$ as well.
  However, if we assumed~$\dim \overline{\phi(Q)}= n-1$ and we considered a~$(g-1-n)$-plane~$L\subset \mathbb{P}^{g-1}$ not meeting~$\overline{\phi(Q)}$, the linear span of~$L$ and~$\overline{\phi(Q)}$ would be the whole~$\mathbb{P}^{g-1}$, a contradiction.

  Thus~$\dim \overline{\phi(Q)}\leqslant n-2$, and the geometric version of the Riemann--Roch theorem yields that~$\dim |Q|=\deg Q- 1-\dim \overline{\phi(Q)}\geqslant 1$, that is~$Q\in C^1_n$.
\end{proof}

\begin{lemma}
  \label{lemma:base-locus-symn-C-2}
  If~$q_1+\cdots+q_n\in C_n^1$, then~$q_1+\cdots+q_n\in\Bs|\omega_{\Sym^nC}|$.
\end{lemma}

\begin{proof}
  Consider the~$n$\dash fold ordinary product~$C^n$ endowed with the natural projections~$\pi_i\colon C^n\to C$, with $i=1,\ldots,n$.
  We follow \cite[\S3-8]{MR0151460}, and we consider a basis~$\left\{\omega_1,\dots,\omega_g\right\}$ of the space~$\HH^0(C,\omega_C)\cong\HH^{1,0}(C)$.
  Then for any $j=1,\ldots,g$, we can define a holomorphic 1-form on~$C^n$
  \begin{equation}
    \xi_j\colonequals\sum_{i=1}^n \pi_i^*\omega_j \in\HH^{1,0}(C^n),
  \end{equation}
  which is invariant under the action of the symmetric group~$\symmetric_n$, so that it can be considered as a holomorphic~1\dash form on the~$n$\dash fold symmetric product~$\Sym^nC$.

  In \cite[\S8]{MR0151460}, Macdonald proved that the set of holomorphic~$n$\dash forms
  \begin{equation}
    \left\{ \xi_{j_1}\wedge \xi_{j_2}\wedge\dots\wedge \xi_{j_n}\mid 1\leqslant j_1<j_2<\dots<j_n\leqslant g \right\}
  \end{equation}
  is a basis for the space~$\HH^{n,0}(\Sym^nC)\cong\HH^0(\Sym^nC,\omega_{\Sym^nC})$ of canonical forms on~$\Sym^nC$.

  Now we consider a point~$Q\colonequals q_1+\cdots+q_n\in C^1_n$ and we assume furthermore that the points~$q_1,\dots,q_n\in C$ are distinct, that is~$Q\in C^1_n\smallsetminus \Delta$, where
  \begin{equation}
    \Delta\colonequals\left\{ 2p+p_3+\cdots+p_n\in \Sym^nC\mid p,p_3,\ldots,p_n\in C \right\}
  \end{equation}
  is the diagonal divisor of $\Sym^nC$. In this case, we may write (with a slight abuse of notation)
  \begin{equation}
    \label{xiQ}
    \xi_j(Q)=\sum_{i=1}^n \omega_j(q_i).
  \end{equation}
  As in the proof of \Cref{lemma:base-locus-symn-C-1}, we consider the canonical embedding~$\phi\colon C\hookrightarrow\mathbb{P}^{g-1}$ and, by the geometric version of the Riemann--Roch theorem, we deduce that the linear span~$\overline{\phi(Q)}$ of~$\phi(q_1),\dots,\phi(q_n)$ has dimension~$\dim \overline{\phi(Q)} = \deg Q- 1-\dim |Q|\leqslant n-1-1=n-2$.

  This is equivalent to saying that the vectors
  \begin{equation}
    (\omega_1(q_i),\dots,\omega_g(q_i)) \quad i=1,\dots,n
  \end{equation}
  are linearly dependent, i.e.~the rank of the~$n\times g$ matrix~$A\colonequals\left(\omega_j(q_i)\right)$ is smaller than~$n$. Thus for any multi-index~$(j_1,j_2,\ldots,j_n)$ such that~$1\leqslant j_1<j_2<\ldots<j_n\leqslant g$, the columns~$(\omega_{j_{k}}(q_1),\dots,\omega_{j_{k}}(q_n))^{\rm T}$ of~$A$ are linearly dependent, and the equality \eqref{xiQ} gives that the canonical form~$\xi_{j_1}\wedge \xi_{j_2}\wedge\dots\wedge \xi_{j_n}$ vanishes at~$Q$. Since these holomorphic~$n$\dash forms provide a basis of~$\HH^0(\Sym^nC,\omega_{\Sym^nC})$, we deduce that if~$Q\in C^1_n$ and the points~$q_1,\dots,q_n\in C$ are distinct, then~$Q\in \Bs|\omega_{\Sym^nC}|$.

  To conclude the proof, we must assume~$Q\in C^1_n\cap\Delta$ and prove that~$Q\in \Bs|\omega_{\Sym^nC}|$. We recall that~$\Bs|\omega_{\Sym^nC}|$ is a closed subset of~$\Sym^nC$ and since~$C^1_n\smallsetminus \Delta\subset \Bs|\omega_{\Sym^nC}|$ by the first part of the proof, it is enough to show that if~$Q\in C^1_n\cap\Delta$, then~$Q$ belongs to the Zariski closure of~$C^1_n\smallsetminus \Delta$. To this aim, let us assume that~$Q=2q+q_3+\cdots+q_{n}$, where~$q,q_3,\dots,q_n$ are distinct points (the case with more points being equal is analogous) and we consider the complete linear series~$|Q|$ on~$C$. Since the divisor~$Q\in \Sym^nC$ is singular at the point~$q\in C$, Bertini's theorem ensures that either the general divisor~$P\in |Q|$ is smooth, or~$2q$ is a base point of the complete linear system~$|Q|$.

  In the former case~$Q$ belongs to the Zariski closure of the locus~$|Q|\smallsetminus \Delta\subset C^1_n$, which is contained in~$\Bs|\omega_{\Sym^nC}|$ by the first part of the proof.

  In the latter case, any divisor in~$|Q|$ has the form~$2q+p_3+\cdots+p_n$, with~$P=p_3+\cdots+p_n$ varying in a linear series~$|P|$ of degree~$n-2$ and dimension equal to~$\dim |Q|$. Therefore, taking a pair of general points~$p_1,p_2\in C$ and a divisor~$D\in |P|$, we have that~$p_1+p_2+D$ lies in~$C^1_n\smallsetminus \Delta$, which is contained in~$\Bs|\omega_{\Sym^nC}|$ by the first part of the proof. Since~$Q$ lies in the Zariski closure of the locus described by the points~$p_1+p_2+D\in C_n^1$ as above, we conclude that~$Q\in \Bs|\omega_{\Sym^nC}|$ as well.

  Finally, the case where~$Q=\sum_{x\in C}n_x x$ has more than two equal points can be handled recursively in the same way, by taking a general divisor~$P=\sum_{x\in C}m_x x\in |Q|$, and distinguishing the following cases: either~$m_x=n_x$ for any~$x\in C$ such that~$n_x>1$, or there exists some~$x\in C$ such that~$n_x>1$ and~$0\leqslant m_x<n_x$ (i.e.~$P$ is less singular than~$Q$ at~$x$).
\end{proof}

This gives an alternative proof of the main result of \cite{1807.10702v1}. The case~$g=2$ here is covered by \cite{MR2838062}.
\begin{proposition}
  \label{proposition:biswas-gomez-lee}
  Let~$C$ be a smooth projective curve of genus~$g\geq 2$. Let~$n=1,\ldots,\gon(C)-1$. Then~$\derived^\bounded(\Sym^nC)$ is indecomposable.
\end{proposition}

\begin{proof}
  By \Cref{proposition:curve-base-locus}\,(1) we have that in this case the canonical base locus~$\Bs|\omega_{\Sym^nC}|$ is empty, so that by \Cref{theorem:kawatani-okawa} we have that~$\derived^\bounded(\Sym^nC)$ is indecomposable.
\end{proof}

With this result in mind, the following proposition is an application of \Cref{corollary:zero-dim-Bs}. Hence our main contribution here is the amplification using \Cref{theorem:indecomposability} from a general curve to \emph{every} curve, removing the dependence of~$n$ on the curve.

\begin{proposition}
  \label{proposition:symmetric-power-indecomposability}
  Let~$C$ be a smooth projective curve of genus~$g\geq 2$.
  Assume \cref{conjecture:ntsod}.
  Then~$\derived^\bounded(\Sym^nC)$ is indecomposable for~$n=1,\ldots,\lfloor\frac{g+3}{2}\rfloor-1$.
\end{proposition}

\begin{proof}
We note that the gonality of $C$ is the least integer $d$ such that $W_d^1(C)$ is non-empty, and $\left\lfloor \frac{g+3}{2}\right\rfloor$ is the least integer $d$ such that the Brill--Noether number $\rho(g,1,d):=g-2(g-d+1)$ is nonnegative.
Thus [1, Theorems V.1.1 and V.1.5] yield that $\gon(C)$ is bounded above by $\left\lfloor \frac{g+3}{2}\right\rfloor$ and a general curve realises this bound. 
  Therefore, thanks to \cite[Theorem, p.~78]{MR927950}, we can construct a smooth family $f\colon\mathcal{C}\to T$ of curves parameterised over a curve $T$ mapping finitely to a complete curve in the moduli space $\mathcal{M}_g$, where $f^{-1}(\overline{t})=C$ for some $\overline{t}\in T$ and the general member of the family is a curve of gonality $\lfloor\frac{g+3}{2}\rfloor$. In a little more detail, one would fix a smooth genus $g$ curve $D$ of maximal gonality $\lfloor\frac{g+3}{2}\rfloor$, and apply \cite[Theorem, p.~78]{MR927950} to obtain a complete curve $T' \subset \mathcal M_g$ passing through $[C]$ and $[D]$. Possibly after an \'etale cover $T \to T'$, this yields the desired family $\mathcal C \to T$.

  Thus the relative symmetric product $f_n\colon\Sym^n_T\mathcal{C}\to T$ is a smooth projective family, whose general fibre $f_n^{-1}(t)$ is the $n$-fold symmetric product of a smooth curve $f^{-1}(t)$ having gonality $\lfloor\frac{g+3}{2}\rfloor$, so that $\derived^\bounded(f_n^{-1}(t))$ is indecomposable by \Cref{proposition:biswas-gomez-lee}.
  Hence the morphism $f_n$ satisfies the assumptions of \Cref{theorem:indecomposability}, and the assertion follows.
\end{proof}

The expected indecomposability for the derived category for~$n\leq g-1$ was in this case
obtained by Lin~\cite{2107.09564} using different (but related) methods
involving the paracanonical base locus,
after the first version of the preprint was released.
Thus,
\cref{proposition:biswas-gomez-lee,proposition:symmetric-power-indecomposability}
are no longer the best known results.
We have included them here as it was our original motivation
for developing the indecomposability technique described in this paper,
and the explicit study of the behaviour of the canonical base locus of~$\Sym^nC$
suggests how to tackle similar questions for other varieties,
cf.~\cref{section:surfaces,section:hilbert-schemes}.

\begin{remark}
  The derived category~$\derived^\bounded(\Sym^nC)$ is \emph{not} indecomposable for~$n\geq g$.
  In \cite[Corollary~5.11]{1805.00183v2} a semiorthogonal decomposition was obtained using wall-crossing methods.
  Remark that~$\Sym^nC$ is also not \emph{birationally minimal} (i.e., the canonical bundle is not nef) in this case. One can check this, e.g., by the arguments of \cite{227059}.

  In fact, the Abel--Jacobi morphism exhibits the symmetric power as a projective bundle
  (of a not necessarily locally free sheaf).
  This description was used in \cite[Theorem~D]{1909.04321v1} and \cite[Corollary~3.8]{1811.12525v4} to give the (same) semiorthogonal decomposition
  \begin{equation}
    \derived^\bounded(\Sym^nC)
    =
    \left\langle
    \derived^\bounded(\Jac C),
    \ldots,
    \derived^\bounded(\Jac C),
    \derived^\bounded(\Sym^{2g-2-n}C)
    \right\rangle
  \end{equation}
  where there are~$n-g+1$ copies of the derived category of the Jacobian.
\end{remark}

\section{Surfaces of general type}
\label{section:surfaces}
To date,
the best indecomposability result for regular surfaces of general type is \cite[Theorem~4.1]{2304.14048}.
The only surfaces for which indecomposability is expected
but not yet known,
are described in \cite[Corollary~4.7]{2304.14048}:
such a surface needs to be minimal,
have~$\HH^1(S,\mathcal{O}_S)=0$ (i.e., it is regular),
$\HH^2(S,\mathcal{O}_S)\neq 0$,
and it must have a~1-dimensional component in the canonical base locus
which cannot be contracted to a point in the category of algebraic spaces.

In this section we will consider two classes of
regular smooth complex projective surfaces of general type
of the sort described in the previous paragraph.
We explain how the canonical base locus for these surfaces varies in families,
in order to understand the indecomposability of their derived categories.

\Cref{subsection:horikawa} concerns the so-called \emph{Horikawa surfaces}.
We summarise what is known about their canonical base loci
and describe the behaviour under degeneration in \Cref{appendix:horikawa}.
\Cref{subsection:ciliberto} is instead concerned with
a series of examples of \emph{double covers of the projective plane} introduced by Ciliberto \cite{MR0778862}.

In both cases we study surfaces $S$ having positive-dimensional canonical base locus,
and our argument for proving the indecomposability of~$\derived^\bounded(S)$
relies on the behaviour of their canonical base loci in families.
This description of the canonical base locus will be in turn used in \Cref{section:hilbert-schemes} to bootstrap indecomposability results for the Hilbert schemes of points $S^{[n]}$.

\subsection{Horikawa surfaces}
\label{subsection:horikawa}
In a series of papers \cite{MR1573789,MR0424831,MR0460340,MR0501370,MR0517773} Horikawa classified minimal surfaces of general type which are on (or close to) the Noether line, i.e.~for which the inequality
\begin{equation}
  \mathrm{c}_1^2\geq 2\pg-4
\end{equation}
is (almost) an equality. In particular, he considered the cases where~$\mathrm{c}_1^2$ equals~$2\pg-4$ \cite{MR0424831}, $2\pg-3$ \cite{MR0460340,MR1573789} or $2\pg-2$ \cite{MR0501370,MR0517773}.
In \Cref{appendix:horikawa}, we summarise the behaviour of the canonical base locus in the classification of Horikawa surfaces.
 It follows from this description that it is possible to realise each Horikawa surface in a family where the general member has empty canonical base locus.
By applying \Cref{theorem:indecomposability}, we thus obtain the following theorem.

\begin{theorem}
  \label{theorem:horikawa-indecomposability}
  Let~$S$ be a Horikawa surface, i.e.~a minimal smooth projective surface of general type with $\mathrm{c}_1^2=2\pg-4$, $2\pg-3$ or~$2\pg-2$.
  If the dimension of the canonical base locus is~1, then furthermore assume \cref{conjecture:ntsod}.
  Then~$\derived^\bounded(S)$ is indecomposable.
\end{theorem}

\begin{remark}
  We note that for Horikawa surfaces with empty or finite canonical base locus,
  the indecomposability of~$\derived^\bounded(S)$ follows immediately from \Cref{theorem:kawatani-okawa}.
  For all the others (i.e.,~Horikawa surfaces whose canonical base locus consists of a rational curve $F$ with $F^2=-2$)
  one can deduce \Cref{theorem:horikawa-indecomposability} alternatively by using \cite[Theorem~1.10]{1508.00682v2},
  which asserts that if~$S$ is a minimal surface such that~$\pg\geq 2$ and every one-dimensional connected component of~$\Bs|\omega_S|$
  has negative definite intersection matrix, then~$\derived^\bounded(S)$ is indecomposable.
  This result is indeed applied in \cite[Example~4.6]{1508.00682v2} to Horikawa surfaces of type~III from \cite{MR0517773}.
  However, we will see in \Cref{section:hilbert-schemes} how our degeneration argument
  is crucial to discuss indecomposability of Hilbert schemes of points of Horikawa surfaces.
\end{remark}


\subsection{Ciliberto double planes}
\label{subsection:ciliberto}


In \cite{MR0778862} Ciliberto considered various classes of surfaces of general type,
which exhibit interesting properties in their canonical coordinate ring.
For our purposes we are interested in \cite[Esempio~4.3]{MR0778862},
where for any integer $h\geq 1$, we obtain a family of minimal smooth projective surfaces $S_h$
of general type of geometric genus $\pg(S_h)=2h$,
whose canonical base locus $\Bs|\omega_{S_h}|$ consists of an irreducible curve $F$ with vanishing self-intersection,
and thus \cite[Theorem~1.10]{1508.00682v2} does \emph{not} apply.

The surfaces are obtained as minimal desingularisations of double covers of $\mathbb{P}^2$
branched along certain plane curves with prescribed singularities.
The existence of the branch curves is ensured by the next proposition,
of which we include a proof for the sake of completeness.

As it is customary, given two integers $a\geq b\geq 0$,
we say that a plane curve $C\subset \mathbb{P}^2$ has a point of type $[a,b]$ at $p\in C$
if the curve has a singularity of multiplicity~$a$ at~$p$
and a singularity of multiplicity $b$ at a point infinitely near to $p$.
In particular, the general curve having a point of type $[a,b]$ at $p$
has exactly $a$ smooth branches passing through $p$,
where $b$ of them are simply tangent to the same line
and the remaining $a-b$ branches have distinct tangent directions.

\begin{proposition}
  \label{proposition:existence-of-curve}
  For every integer~$h\geq 1$ there exists an irreducible curve~$C_h\subset\mathbb{P}^2$ of degree~$8h+4$ having
  \begin{itemize}
    \item an ordinary singularity of multiplicity~$4h+2$ at $p_1$,\label{condition2}
    \item an ordinary singularity of multiplicity~$4h$ at $p_2$,\label{condition3}
    \item four singularities of type~$[2h+1,2h+1]$ at $p_3,\ldots,p_6$,\label{condition4}
  \end{itemize}
 and no other singularities, where the points~$p_1,\ldots,p_6$ lie on an irreducible conic~$\Gamma\subset \mathbb{P}^2$ transverse to $C_h$ at any $p_i$.
\end{proposition}

\begin{proof}
Let $\Delta\subset \mathbb{P}^2$ be an irreducible quartic having two nodes at $p_1,p_2\in \mathbb{P}^2$.
Let $\Gamma\subset \mathbb{P}^2$ be an irreducible conic passing through $p_1,p_2$ and intersecting $\Delta$ transversely at four other points $p_3,\dots,p_6$.
Let $\mathcal{L}_h$ be the linear system of curves of degree $8h+4$ having (at least) a $(4h+2)$-tuple point at $p_1$, a $4h$-tuple point at $p_2$, and points of type $[2h+1,2h+1]$ at $p_3,\dots,p_6$ with tangent directions given by the tangent lines of $\Delta$ at $p_3,\dots,p_6$.
We point out that $\mathcal{L}_h$ contains also curves having worse singularities, which are obtained by degeneration of the singularities at $p_1,\dots,p_6$. 

We need to prove that $\mathcal{L}_h$ contains irreducible curves.
To this aim we define two degree $8h+4$ curves $A:= (2h+1)\Delta$ and $B:= 4h\Gamma+\Sigma+\ell$, where $\ell$ is a general line through $p_1$, and $\Sigma$ is a cubic passing through $p_1$ and tangent to $\Delta$ at $p_3,\dots,p_6$, which exists by a dimension count.
Therefore, $A$ has points of multiplicity $4h+2$ at $p_1$ and $p_2$, and points of type $[2h+1,2h+1]$ at $p_3,\dots,p_6$, so that $A\in \mathcal{L}_h$. 
As far as the curve $B$ is concerned, it possesses a $(4h+2)$-tuple point at $p_1$ and a $4h$-tuple point at $p_2$.
Moreover, let $\ell_i$ denotes the tangent line to $\Delta$ at $p_i$, with $i=3,\dots,6$.
Then each $\ell_i$ is tangent to $\Sigma$ at $p_i$, and it meets the non-reduced quartic $2\Gamma$ at $p_i$ with multiplicity 2.
Since $B$ contains $2h$ copies of $2\Gamma$, we conclude that $B$ has (a degeneration of) a point of type $[2h+1,2h+1]$ at $p_3,\dots,p_6$ with tangent directions given by $\ell_3,\dots,\ell_6$, so that $B\in \mathcal{L}_h$.

Then we consider the pencil $\Phi:=\alpha A+\beta B$ of curves of $\mathcal{L}_h$. Since $\Phi$ has no fixed curves, Bertini's second theorem implies that the general curve of the pencil is reducible if and only if $\Phi$ is composed with a pencil, i.e. if there exists a pencil~$\Psi=\lambda C+\mu C'$ and~$n\geq 2$ such that~$A=C_1\cup\cdots\cup C_n$ and~$B=C_1'\cup\cdots\cup C'_n$ for some~$C_j,C_j'\in\Psi$.
Suppose by contradiction that this is the case.
Hence we have that $C_1=k\Delta$ for some integer $k$, and we can set $C_1'$ to be the component of $B$ containing $\ell$.
Then $\Psi=\lambda C_1+\mu C'_1$, and any curve of $\Psi$ must pass through $C_1\cap C'_1$.
However, $C'_1$ is the \emph{unique} component of $B$ passing through all the points of $\Delta\cap\ell$, a contradiction.

Furthermore, the general curve of $\Phi$ is smooth at the base points outside $p_1,\ldots,p_6$, because $B$ is.
Thus we conclude by Bertini's theorem that $\mathcal{L}_h$ contains irreducible curves having singularities only at $p_1,\dots,p_6$, and the assertion follows.
\end{proof}

Fixing an integer $h\geq 1$, we consider a curve $C_h\subset \mathbb{P}^2$ as in \Cref{proposition:existence-of-curve} and the double plane $\pi'\colon S_h'\to \mathbb{P}^2$ branched along $C_h$.
Then we define the surface $S_h$ to be the minimal desingularisation of $S'_h$,
which is endowed with a generically finite morphism of degree 2 induced by $\pi'$,
\begin{equation}
  \pi\colon S_h\to \mathbb{P}^2.
\end{equation}

\begin{remark}
\label{remark:K_S_h}
The image under $\pi$ of the canonical linear series $|\omega_{S_h}|$ is the linear system $\mathcal{D}_h$ of curves $D_h\subset \mathbb{P}^2$ of degree $4h-1$ having a $2h$-tuple point at $p_1$, a $(2h-1)$-tuple point at $p_2$, and points of type $[h,h-1]$ at $p_3,\dots,p_6$ (cfr. e.g. \cite[Theorem III.7.2]{MR2030225}).
Therefore, the conic $\Gamma$ is a fixed component of $\mathcal{D}_h$ by B\'ezout's theorem.
Thus the curve $F\colonequals \pi^{-1}(\Gamma)$ is a fixed curve of $|\omega_{S_h}|$, which can be proved to be a smooth irreducible curve of genus $2$ with $F^2=0$.

So the curves in $\mathcal{D}_h$ are of the form $D_h=\Gamma + E_h$, where $E_h\subset \mathbb{P}^2$ varies in the linear system $\mathcal{E}_h$ of curves of degree $4h-3$ with a $(2h-1)$-tuple point at $p_1$, a $(2h-2)$-tuple point at $p_2$, and points of type $[h-1,h-1]$ at $p_3,\dots,p_6$.
Therefore, counting conditions imposed by the prescribed singularities of $E_h$, we deduce
\begin{equation}
\label{eq:dim E_h}
\dim \mathcal{E}_h\geq\frac{(4h-3)4h}{2}-\frac{(2h-1)2h}{2}-\frac{(2h-2)(2h-1)}{2}-4\cdot2\frac{(h-1)h}{2}=2h-1,
\end{equation}
and noting that the curves of $\mathcal{E}_h$ passing through two extra points $p_7,p_8\in \Gamma$ are of the form $2\Gamma + E_{h-1}$, one concludes by induction on $h$ that \eqref{eq:dim E_h} is actually an equality.
  Thus the geometric genus of $S_h$ is $\pg(S_h)=\dim \mathcal{D}_h +1=\dim \mathcal{E}_h +1=2h$,
  whereas the irregularity is $\qq(S_h)=0$ (see e.g. \cite[Section V.22]{MR2030225}).
\end{remark}

Finally, we consider the linear system $\mathcal{L}_h$ appearing in the proof of \Cref{proposition:existence-of-curve}, consisting of plane curves of degree $8h+4$ having a $(4h+2)$-tuple point at $p_1$, a $4h$-tuple point at $p_2$, and points of type $[2h+1,2h+1]$ at $p_3,\dots,p_6$ (with fixed tangent directions), where $p_1,\dots,p_6$ lie on the irreducible conic $\Gamma$.
In particular, we compute the dimension of $\mathcal{L}_h$, which is involved in the degeneration argument of \Cref{subsection:ciliberto}.

\begin{lemma}
\label{lemma:dimension-L}
Given an integer $h\geq 1$ and using the notation above, the linear system $\mathcal{L}_h$ has dimension $8h+3$.
\end{lemma}

\begin{proof}
Arguing as in \eqref{eq:dim E_h}, we note that
\begin{equation}
\begin{aligned}
\label{eq:dim L_h}
\dim \mathcal{L}_h & \geq\frac{(8h+4)(8h+7)}{2}-\frac{(4h+2)(4h+3)}{2}-\frac{4h(4h+1)}{2}\\
& -4\cdot2\frac{(2h+1)(2h+2)}{2}
 =8h+3.
\end{aligned}
\end{equation}

In order to prove that \eqref{eq:dim L_h} is actually an equality, we consider a general element $C=C_h$ of $\mathcal{L}_h$ as in \Cref{proposition:existence-of-curve} and the blowup~$\widetilde{X}\colonequals \Bl_{p_1,\ldots,p_6}\mathbb{P}^2$ of~$\mathbb{P}^2$ at the points~$p_1,\ldots,p_6$, with exceptional divisors by~$\widetilde{E}_1,\ldots,\widetilde{E}_6$.
Let $q_3,\dots,q_6$ be the points on~$\widetilde{E}_3,\ldots,\widetilde{E}_6$ corresponding to the common tangent directions of the~$2h+1$ branches of~$C$ at $p_3,\dots,p_6$, and let~$X\colonequals \Bl_{q_3,\ldots,q_6}\widetilde{X}$ be the blowup at these points, with exceptional divisors by~$F_3,\ldots,F_6$.
We denote by $E_1,\dots,E_6\subset X$ the strict transforms of $\widetilde{E}_1,\ldots,\widetilde{E}_6$, so that $E_1^2=E_2^2=-1$, $E_3^2=\dots=E_6^2=-2$ and $E_i\cdot F_j=\delta_{ij}$.

Let $\mathcal{L}_h^*$ be the linear system obtained on $X$ by taking strict transforms of the curves in $\mathcal{L}_h$, and let $D\in \mathcal{L}_h^*$ be the strict transform of $C$, which is a smooth irreducible curve, as the singularities of $C$ have been resolved by the sequence of blowups.
Denoting by $H$ the class of the strict transform of a line in $\mathbb{P}^2$, the numerical equivalence classes of a canonical divisor $\mathrm{K}_X$ and $D$ can be easily computed to be
\begin{equation}
  \begin{aligned}
    \mathrm{K}_X&\equiv -3H+\sum_{i=1}^6E_1+2\sum_{j=3}^6F_j \\
    D&\equiv (8h+4)H-(4h+2)E_1-4hE_2-(2h+1)\sum_{j=3}^6E_j-2(2h+1)\sum_{j=3}^6F_j.
  \end{aligned}
\end{equation}
Hence~$D^2=16h+4$ and~$D\cdot\mathrm{K}_X=-2$, so that $g(D)=1+\frac12(D^2+D\cdot\mathrm{K}_X)=8h+2$ by the adjunction formula.

Finally, we consider the characteristic linear series~$\mathcal{L}_h^*|_D$ cut out on~$D$ by the other curves in~$\mathcal{L}_h^*$, so that
\begin{equation}
\deg \mathcal{L}_h^*|_D=D^2=16h+4=2g(D) \quad\text{and}\quad \dim \mathcal{L}_h^*|_D=\dim  \mathcal{L}_h^*-1\geq 8h+2
\end{equation}
by \eqref{eq:dim L_h}.
On the other hand, the Riemann--Roch theorem yields that
\begin{equation}
\dim \mathcal{L}_h^*|_D\leq \deg \mathcal{L}_h^*|_D-g(D)=8h+2.
\end{equation}
Thus we conclude that $\dim \mathcal{L}_h=\dim \mathcal{L}_h^*=\dim  \mathcal{L}_h^*|_D+1= 8h+3$, as claimed.
\end{proof}

\begin{theorem}
  \label{theorem:ciliberto-family}
  Let~$h\geq 1$ be an integer and let $C_h\subset \mathbb{P}^2$ be an irreducible curve of degree $8h+4$, having a $(4h+2)$-tuple ordinary point at $p_1$, a $4h$-tuple ordinary point at $p_2$, four points of type $[2h+1,2h+1]$ at $p_3,\dots,p_6$, and no other singularities, where $p_1,\dots,p_6\in \mathbb{P}^2$ lie on an irreducible conic transverse to $C_h$ at any $p_i$.
 Let~$S_h$ be the minimal desingularisation of the double plane~$S'_h\to \mathbb{P}^2$ branched along~$C_h$.
 Then~$S_h$ can be put in a family
 such that the canonical base locus of a general element has dimension at most~0.
\end{theorem}

\begin{proof}
  We want to construct a smooth family $f\colon\mathcal{X}\to T$ of projective surfaces such that $\dim \Bs|\omega_{f^{-1}(t)}|\leq 0$ and $f^{-1}(t')=S_h$ for some $t,t'\in T$.

  Let $Y\subset \mathbb{P}^2\times \mathbb{G}(1,2)$ be
  the universal family of the Grassmannian of lines in $\mathbb P^2$,
  and let $U\subset \mathbb{P}^2\times \mathbb{P}^2\times \Sym^4 Y$
  be the irreducible open subset parametrising tuples $u=\left(x_1,x_2,(x_3,[\ell_3])+\dots+(x_6,[\ell_6])\right)$
  such that the points $x_1,\dots,x_6$ are distinct,
  no three of them lie on a line,
  and there exists an irreducible quartic with nodes at $x_1,x_2$
  and tangent to each $\ell_i$ at $x_i$.\footnote{The latter is an open condition on $\mathbb{P}^2\times \mathbb{P}^2\times \Sym^4 Y$
    which ensures that,
    although the points $x_1,\dots,x_6$ may lie on an irreducible conic,
    the tangent directions are sufficiently general,
    as in \Cref{proposition:existence-of-curve}.
  }
For any $u\in U$, let $\mathcal{L}_u\subset \mathbb{P}\left(\HH^0\left(\mathbb{P}^2,\mathcal{O}_{\mathbb{P}^2}(8h+4)\right)\right)$ be the linear system of plane curves of degree $8h+4$ having a $(4h+2)$-tuple point at $x_1$, a $4h$-tuple point at $x_2$, four points of type $[2h+1,2h+1]$ at $x_3,\dots,x_6$, each with tangent line $\ell_i$.
Consider the incidence variety
\begin{equation}
  W\colonequals\left\{ ([F],u)\in \mathbb{P}\left( \HH^0(\mathbb{P}^2,\mathcal{O}_{\mathbb{P}^2}(8h+4)) \right)\times U \mid C=\mathrm{V}(F)\in \mathcal{L}_u \right\},
\end{equation}
For any $u\in U$ the fibre via the second projection $\pi_2\colon W\to U$ is $\pi_2^{-1}(u)\cong \mathcal{L}_u$.
As in \eqref{eq:dim L_h}, $\dim \mathcal{L}_u\geq 8h+3$ and \Cref{lemma:dimension-L} ensures that this is an equality when $x_1,\dots, x_6$ lie on an irreducible conic.
Then we deduce by semi-continuity that $\dim \mathcal{L}_u= 8h+3$ also for points $x_1,\dots, x_6$ in general position.
Being $\pi_2$ a surjective map with fibres $\pi_2^{-1}(u)\cong \mathbb{P}^{8h+3}$ over the irreducible variety $U$, we conclude that $W$ is itself irreducible.
Therefore, for general $w=\left([F],u\right)\in W$, the curve $C_w\colonequals\mathrm{V}(F)\in \mathcal{L}_u$ is irreducible and its only singular points are a $(4h+2)$-tuple ordinary point at $x_1$, a $4h$-tuple ordinary point at $x_2$, and four points of type $[2h+1,2h+1]$ at $x_3,\dots,x_6$, because this happens when the points $x_1,\dots,x_6$ are special, as we check in \Cref{proposition:existence-of-curve}.

Let $T\subset W$ be the locus parametrising irreducible curves $C_w$ as above.
Then we can define a family of surfaces $f'\colon \mathcal{X}'\to T$ such that for any $t\in T$, the fibre $X'_t=(f')^{-1}(t)$ is the double covering of $\mathbb{P}^2$ branched over the curve $C_t$.
Up to base-changing and shrinking $T$, we then obtain a family $f\colon\mathcal{X}\to T$ of smooth surfaces of general type, where $X_t=f^{-1}(t)$ is a minimal desingularisation of $X'_t$.
In particular, when the singularities of the curve $C_t$ lie on an irreducible conic, we obtain a surface $S_h$ as in the assertion.
Thus it remains to show that there exists $t\in T$ such that $\Bs|\omega_{X_t}|$ does not contain a curve.

To this aim, we consider a general point $t\in T$ and we set $u\colonequals \pi_2(t)$. 
Let $\phi\colon X_t\to \mathbb{P}^2$ be the generically finite morphism of degree 2 induced by the double covering $X'_t\to \mathbb{P}^2$.
As in \Cref{remark:K_S_h}, the image under $\phi$ of the curves in the canonical linear system $|\omega_{X_t}|$ is the linear system $\mathcal{D}_h$ of curves $D\subset \mathbb{P}^2$ of degree $4h-1$ having a $2h$-tuple point at $x_1$, a $(2h-1)$-tuple point at $x_2$ and four points of type $[h,h-1]$ at $x_3,\dots,x_6$.
Hence $|\omega_{X_t}|$ possesses a fixed curve if and only if $\mathcal{D}_h$ does.

However, if $h=1$, then $\mathcal{D}_1$ is the pencil of cubics having a node at $x_1$ and passing through $x_2,\dots,x_6$, whose general element is indeed an irreducible curve.

When $h\geq 2$, we argue as in \eqref{eq:dim E_h}, and we deduce
\begin{equation}
\begin{aligned}
\dim \mathcal{D}_h \geq&\frac{(4h-1)(4h+2)}{2}-\frac{2h(2h+1)}{2}-\frac{(2h-1)2h}{2} \\ & - 4\left(\frac{h(h+1)}{2}+\frac{(h-1)h}{2}\right)=2h-1.
\end{aligned}
\end{equation}
On the other hand, by \Cref{remark:K_S_h}, the latter is an equality when the points $x_1,\ldots,x_6$ lie on an irreducible conic, hence $\dim\mathcal{D}_h= 2h-1$ by semicontinuity.
Moreover, we recall that for such a general $u= \left(x_1,x_2,(x_3,[\ell_3])+\dots+(x_6,[\ell_6])\right)$, there exists an irreducible quartic $\Delta\subset \mathbb{P}^2$ having nodes at $x_1,x_2$ and passing through $x_3,\dots,x_6$ with tangent lines $\ell_3,\dots,\ell_6$.
Therefore, for any cubic $D\in \mathcal{D}_1$ as above, the curve $D' \colonequals (h-1)\Delta+D$ belongs to $\mathcal{D}_h$.
Hence $\mathcal{D}_h$ has a fixed curve if and only if $\Delta\subset D_h$ for any $D_h\in \mathcal{D}_h$.
If this were the case, residual curves would be such that $\overline{D_h\setminus\Delta}\in \mathcal{D}_{h-1}$, and hence $2h-1=\dim \mathcal{D}_h=\dim\mathcal{D}_{h-1}=2(h-1)-1$,  a contradiction.
Then $\mathcal{D}_h$ has no fixed curves for any $h\geq 2$.
\end{proof}

\begin{corollary}
  \label{corollary:ciliberto-indecomposability}
  Let~$S_h$ be as in \cref{theorem:ciliberto-family}.
  and assume that \cref{conjecture:ntsod} holds.
  Then~$\derived^\bounded(S_h)$ is indecomposable.
\end{corollary}

\begin{proof}
  For the family in the proof of \cref{theorem:ciliberto-family}
  we conclude that for generic~$t$ the canonical base locus $\Bs|\omega_{X_t}|$ does not contain curves,
  and the assertion follows from \Cref{corollary:zero-dim-Bs}.
\end{proof}

\begin{remark}
Let $S_h$ be a surface as in \Cref{theorem:ciliberto-family}.
It follows from the analysis of \cite[Esempio~4.3]{MR0778862} that the fixed part of $|\omega_{S_h}|$ is a genus 2 curve $F$ such that $F^2=0$.
Therefore, unlike the case of Horikawa surfaces, it is not possible to apply \cite[Theorem~1.10]{1508.00682v2} to achieve the indecomposability of $\derived^\bounded(S_h)$, and we do not know alternative arguments for proving \Cref{theorem:ciliberto-family}.
\end{remark}

\begin{remark}\label{remark:ciliberto-empty}
  By \cref{theorem:ciliberto-family} the indecomposability conclusion in \cref{corollary:ciliberto-indecomposability}
  holds even for surfaces $S_h$ obtained when the points $p_1,\dots,p_6\in \mathbb{P}^2$ are in general position.
  Indeed, they are the surfaces $X_t$ with $\dim\Bs|\omega_{X_t}|\leq 0$
  appearing in the proof, so that $\derived^\bounded(X_t)$ is indecomposable by \Cref{theorem:kawatani-okawa}.
  Furthermore, one can easily check that for general $t\in T$, the canonical base locus $\Bs|\omega_{X_t}|$ is empty;
  we will use this fact in the next section to deduce that the Hilbert scheme of points $\hilbn{n}{S_h}$
  of the surface $S_h$ has indecomposable derived category.
\end{remark}

\section{Hilbert schemes of points on surfaces}
\label{section:hilbert-schemes}
Given a smooth projective surface $S$, let $\hilbn{n}{S}$ denote its Hilbert scheme of points.
It provides an interesting smooth projective variety (of dimension~$2n$),
obtained as a crepant resolution of the symmetric power of the surface~$S$.
An affirmative answer to \Cref{question:indecomposability}
suggests that for a surface~$S$ with~$\omega_S$ nef and~$\HH^0(S,\omega_S)\neq 0$
both~$\derived^\bounded(S)$
and~$\derived^\bounded(\hilbn{n}{S})$ for~$n\geq 2$ are indecomposable.
The following proposition is our starting point for the study of indecomposability of the Hilbert schemes of points on surfaces.
\begin{proposition}
  \label{proposition:canonical-base-locus-hilbert-schemes}
  Let~$S$ be a smooth projective surface. Let~$n\geq 1$.
  If~$\Bs|\omega_S|=\emptyset$, then~$\Bs|\omega_{\hilbn{n}{S}}|=\emptyset$.
\end{proposition}

\begin{proof}
  The Hilbert--Chow morphism
  \begin{equation}
    \pi\colon\hilbn{n}{S}\to\Sym^nS
  \end{equation}
  is a crepant resolution of singularities, and~$\Sym^nS$ is Gorenstein.
  Hence there is an isomorphism of line bundles~$\omega_{\hilbn{n}{S}}\cong\pi^*\omega_{\Sym^nS}$,
  and combined with the projection formula and the fact that~$\pi_*\mathcal{O}_{\hilbn{n}{S}}\cong \mathcal{O}_{\Sym^nS}$,
  it induces the natural isomorphism
  \begin{equation}
    \pi^*\colon
    \HH^0(\Sym^nS,\omega_{\Sym^nS})
   \to
    \HH^0(\hilbn{n}{S},\omega_{\hilbn{n}{S}}).
  \end{equation}
  This in particular implies the following identification of closed subschemes of~$\hilbn{n}{S}$
  \begin{equation}\label{eqn:PB_BS}
    \pi^{- 1}(\Bs|\omega_{\Sym^nS}|)
    =
    \Bs|\omega_{\hilbn{n}{S}}|.
  \end{equation}
  To compute the canonical base locus of~$\Sym^nS$ we can use the identification
  \begin{equation}
    \HH^0(\Sym^nS,\omega_{\Sym^nS})
    \cong
    \HH^0(S^n,\omega_{S^n})^{\symmetric_n}
  \end{equation}
  induced by the isomorphism of line bundles~$q^*\omega_{\Sym^nS}\cong\omega_{S^n}$, where~$(-)^{\symmetric_n}$ are the~$\symmetric_n$\dash invariant sections, and~$q\colon S^n\to\Sym^nS$ is the quotient map. For each section~$t \in \HH^0(\Sym^nS,\omega_{\Sym^nS})$, the equality~$q^{-1}(\mathrm{V}(t))=\mathrm{V}(q^*(t))$ of closed subschemes of~$S^n$ holds.

  By assumption~$\Bs|\omega_S|=\emptyset$. Let~$z\in\Sym^nS$, and let~$(p_1,\ldots,p_n) \in S^n$ be a point such that~$q(p_1,\ldots,p_n) = z$. A general section~$s\in\HH^0(S,\omega_S)$ avoids~$p_1,\ldots,p_n$ hence, if~$\pi_i\colon S^n\to S$ denotes the projection onto the $i$th factor, the~$\symmetric_n$\dash invariant section $s\otimes\cdots\otimes s=\pi_1^*(s)\otimes\cdots\otimes\pi_n^*(s)$ of~$\omega_{S^n}$ is non-zero at~$(p_1,\ldots,p_n)$. This implies that the corresponding global section of~$\omega_{\Sym^nS}$ does not vanish at~$z$. Hence~$\Bs|\omega_{\Sym^nS}|=\emptyset$, and we are done by \eqref{eqn:PB_BS}.
\end{proof}

\begin{remark}
  This result is to be contrasted with the behaviour of the canonical base locus for symmetric powers of curves, where the power~$n$ plays an important role. The difference between the two cases is that the action of the symmetric group on~$C^n$ has as its non-free locus the big diagonal, which has codimension~1. Hence the quotient morphism~$C^n\to\Sym^nC$ has a ramification divisor. On the other hand, the action on~$S^n$ gives rise to a quotient morphism~$S^n\to\Sym^nS$ which is \'etale in codimension~1, as the non-free locus has codimension~2.
\end{remark}

The description of the canonical base locus in \Cref{proposition:canonical-base-locus-hilbert-schemes} shows indecomposability of the derived category of the Hilbert scheme $\hilbn{n}{S}$ for any surface~$S$ with~$\Bs|\omega_S|=\emptyset$.
There are several classes of surfaces of general type for which this is the case.

When instead $\Bs|\omega_S|\neq \emptyset$, we can try to understand indecomposability of~$\derived^\bounded(\hilbn{n}{S})$
by combining \Cref{proposition:canonical-base-locus-hilbert-schemes} and \Cref{corollary:zero-dim-Bs},
and arguing by degeneration.
Namely, suppose we have a smooth family $f\colon \mathcal{X}\to T$ of projective surfaces such that $f^{-1}(t_0)=S$
and $\Bs|\omega_{f^{-1}(t_1)}|= \emptyset$ for some $t_0,t_1\in T$, and let $f_n\colon \hilbn{n}{\mathcal{X}}\to T$ be the relative Hilbert scheme.
Then \Cref{proposition:canonical-base-locus-hilbert-schemes} ensures that the canonical base locus of the fibre $f_n^{-1}(t_1)$ is empty,
and we conclude by \Cref{corollary:zero-dim-Bs} that $f_n^{-1}(t)$ has indecomposable derived category for any $t\in T$.
One instance of this is given by the following.



\begin{proposition}
\label{proposition:hilbert-horikawa}
  Let~$S$ be a Horikawa surface in one of the following cases:
  \begin{itemize}
    \item for $\mathrm{c}_1^2=2\pg-4$ and~$\pg\geq 3$;
    \item for $\mathrm{c}_1^2=2\pg-3$ and~$\pg=3$;
    \item for $\mathrm{c}_1^2=2\pg-2$ and either $\pg=3$, or~$\pg=4$ and types~$\mathrm{I}, \mathrm{II}, \mathrm{III}_b$, $\mathrm{IV}_{b\mhyphen1}$ or~$\mathrm{V}_2$.
  \end{itemize}
  Then for all~$n\geq 2$
  and assuming \cref{conjecture:ntsod}
  we have that~$\derived^\bounded(\hilbn{n}{S})$ is indecomposable.
\end{proposition}

\begin{proof}
  For~$\mathrm{c}_1^2=2\pg-4$ the assertion follows from \Cref{proposition:canonical-base-locus-hilbert-schemes} and \Cref{theorem:kawatani-okawa}.

  For~$\mathrm{c}_1^2=2\pg-3$ we apply \Cref{corollary:zero-dim-Bs} to the family the Hilbert schemes associated to a family of Horikawa surfaces whose generic member is of type~I, whose canonical base locus is empty.

  For~$\mathrm{c}_1^2=2\pg-2$ and~$\pg=3$ we use that surfaces of type~I have empty canonical base locus, so that \Cref{corollary:zero-dim-Bs} can be applied to a family whose generic member is of this type. The same argument works for~$\pg=4$, by inspecting the degeneration diagrams \eqref{equation:horikawa-III-1}, \eqref{equation:horikawa-III-2}.
\end{proof}

The `types' mentioned in the previous statement, as well as their relationships in terms of degeneration, are recalled in \Cref{appendix:horikawa}.

\begin{remark}
In the light of \Cref{remark:ciliberto-empty}, we may argue analogously in order to achieve indecomposability of~$\derived^\bounded(\hilbn{n}{S_h})$, when $S_h$ is a double covering of $\mathbb{P}^2$ as in \Cref{subsection:ciliberto}.
\end{remark}


\appendix

\section{On Horikawa surfaces}
\label{appendix:horikawa}
In this appendix we summarise the classification of Horikawa surfaces, the description of their canonical base loci, and their behaviour under degeneration.
This is used to show that all Horikawa surfaces have indecomposable derived category (\Cref{theorem:horikawa-indecomposability}), and to deduce indecomposability for Hilbert schemes of points on some classes of Horikawa surfaces (\Cref{proposition:hilbert-horikawa}).

\textbf{The case $\mathrm{c}_1^2=2\pg-4$.}
For surfaces \emph{on} the Noether line, we have by \cite[Lemma~1.1]{MR0424831} that the canonical base locus~$\Bs|\omega_S|$ is empty.

\textbf{The case $\mathrm{c}_1^2=2\pg-3$.}
The canonical base locus now depends on the value of~$\pg\geq 2$, and its description is summarised as follows:
\begin{itemize}
  \item if~$\pg=2$, then as in the proof of \cite[Lemma~2.1]{MR0460340} the canonical base locus consists of a single point;
  \item if~$\pg=3$, then by \cite[Theorem~2.3]{MR0460340} the canonical base locus is either empty (type~I) or consists of a single point (type~II), and by \cite[Lemma~5.3]{MR0460340} we know that every surface of type~II can be obtained as deformation of surfaces of type~I;
  \item if~$\pg=4$, then by \cite[Theorem~1]{MR1573789} the canonical base locus is either empty or consists of a single point;
  \item if~$\pg\geq 5$, then by \cite[Theorem~1.3]{MR0460340} the canonical base locus consists of a single point.
\end{itemize}

For use in \Cref{section:hilbert-schemes} we highlight the behaviour for~$\pg=3$, where we have a deformation from empty canonical base locus to non-empty canonical base locus.

\textbf{The case $\mathrm{c}_1^2=2\pg-2$.}
This is the most interesting case for our purposes. We do not give an exhaustive description, but rather focus on the cases~$\pg=2$ and~$\pg=4$ because the former appears in \cite[Example~4.6]{1508.00682v2} and the latter features all interesting behaviour for the canonical base locus under deformation. If~$A$ and~$B$ denote classes of surfaces, then~$A\leadsto B$ means that any surface of type $B$ is the special member of some smooth family of surfaces of type $A$.
\begin{itemize}
  \item If~$\pg=2$, there are three types, I, II and III. By the discussion in \cite[\S1]{MR0517773} the canonical base locus is either at most two points (type~I or~II) or a rational curve of self-intersection~$-2$ (type~III).

    By \cite[Theorem~1.5(ii)]{MR0517773} we know that every surface of type~III can be obtained as deformation of surfaces of type~II, which in turn are obtained from surfaces of type~I.

  \item If~$\pg=4$, there exist several types depending on the behaviour of the canonical morphism and its indeterminacy locus, which we have summarised in the following two diagrams. These combine \cite[Theorems~4.1, 6.1 and~6.2]{MR0501370} and \cite[Theorem~0.1]{MR2244379}. The canonical base locus can be empty, consist of two distinct points, consist of a double point or have fixed part~$F$ which is a rational curve of self-intersection~$-2$.

    
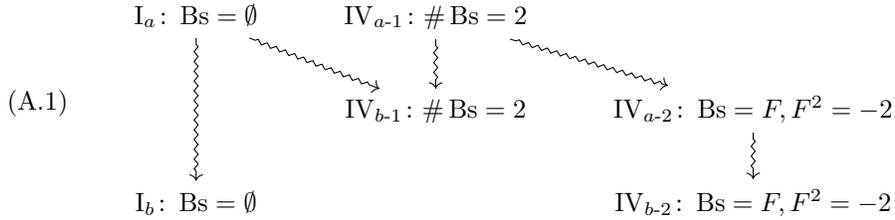
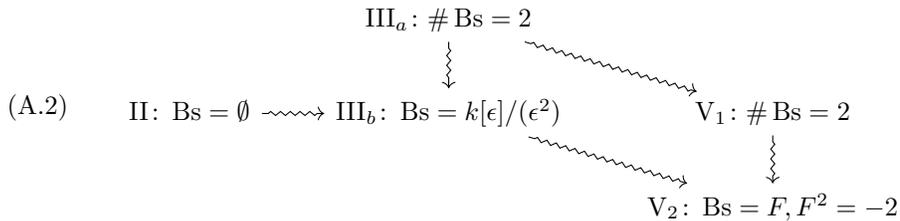
\begin{figure}[ht]
      \begin{equation}
        \label{equation:horikawa-III-1}
        \begin{tikzcd}
          \mathrm{I}_a\colon\Bs=\emptyset \arrow[dd, squiggly] \arrow[rd, squiggly] & \mathrm{IV}_{a\mhyphen1}\colon\#\Bs=2 \arrow[d, squiggly] \arrow[rd, squiggly] \\
          & \mathrm{IV}_{b\mhyphen1}\colon\#\Bs=2 & \mathrm{IV}_{a\mhyphen2}\colon\Bs=F,F^2=-2 \arrow[d, squiggly] \\
          \mathrm{I}_b\colon\Bs=\emptyset & & \mathrm{IV}_{b\mhyphen2}\colon\Bs=F,F^2=-2
        \end{tikzcd}
      \end{equation}

      \begin{equation}
        \label{equation:horikawa-III-2}
        \begin{tikzcd}
          & \mathrm{III}_a\colon\#\Bs=2 \arrow[d, squiggly] \arrow[rd, squiggly] \\
          \mathrm{II}\colon\Bs=\emptyset \arrow[r, squiggly] & \mathrm{III}_b\colon\Bs=k[\epsilon]/(\epsilon^2) \arrow[rd, squiggly] & \mathrm{V}_1\colon\#\Bs=2 \arrow[d, squiggly] \\
          & & \mathrm{V}_2\colon\Bs=F,F^2=-2
        \end{tikzcd}
      \end{equation}
      \caption{Degeneration patterns for base loci of Horikawa surfaces}
    \end{figure}
\end{itemize}
For~$\pg=3$ the analysis started in \cite[\S2]{MR0517773} is completed in \cite{reid,dicks}, giving rise to~4~types of canonical base locus. Using the notation from \cite[\S4]{reid}, type~I has empty canonical base locus, type II a single point, type III two distinct points and type~$\mathrm{III}_a$ a rational curve~$F$ with~$F^2=-2$. There exist degenerations~$\mathrm{I}\leadsto\mathrm{II}$ and~$\mathrm{I}\leadsto\mathrm{III}\leadsto\mathrm{III}_a$. This suffices to perform an analysis similar to the one for~$\pg=4$. For~$\pg\geq5$ one is referred to \cite[\S5, \S6]{MR0517773} for the proofs that the canonical base locus is either finite, or has fixed part a rational curve of self-intersection~$-2$, and the degeneration hierarchies.

In \Cref{section:hilbert-schemes} the deformation (for~$\pg=4$) of type~II into type~$\mathrm{III}_b$ (obtained in \cite[Theorem~0.1]{MR2244379}) and of type~$\mathrm{I}_a$ into type~$\mathrm{IV}_{b\mhyphen1}$, deforming from empty base locus to non-empty base locus, will be used.

\section*{Acknowledgements}

We are grateful to Ciro Ciliberto for helpful discussions and for patiently explaining some details of the examples we described in \Cref{subsection:ciliberto}.
We would also like to thank Michele Graffeo, Kazuhiro Konno and Filippo Viviani for useful and interesting discussions, and Najmuddin Fakhruddin for \Cref{remark:Fakhruddin}.

Finally we want to thank the Max Planck Institute for Mathematics, the University of Bonn, the Tata Institute for Mathematics, and SISSA for the pleasant working conditions.

\bibliographystyle{amsplain}
\bibliography{clean.bib}

\end{document}